\documentclass[a4paper]{amsart}

\usepackage{amsmath}
\usepackage{amsfonts}
\usepackage{amssymb}

\newtheorem{theorem}{Theorem}
\newtheorem{claim}[theorem]{Claim}

\newtheorem{lemma}[theorem]{Lemma}
\newtheorem{proposition}[theorem]{Proposition}

\theoremstyle{definition}
\newtheorem{definition}[theorem]{Definition}

\theoremstyle{remark}
\newtheorem{remark}[theorem]{Remark}

\numberwithin{theorem}{section}
\numberwithin{equation}{section}

\newcommand{\dd}{\; \mathrm{d}}

\def\Xint#1{\mathchoice
{\XXint\displaystyle\textstyle{#1}}
{\XXint\textstyle\scriptstyle{#1}}
{\XXint\scriptstyle\scriptscriptstyle{#1}}
{\XXint\scriptscriptstyle\scriptscriptstyle{#1}}
\!\int}
\def\XXint#1#2#3{{\setbox0=\hbox{$#1{#2#3}{\int}$}
\vcenter{\hbox{$#2#3$}}\kern-.6\wd0}}

\def\dashint{\Xint-}

\begin{document}

\title{Surfaces meeting porous sets in positive measure}
\author{Gareth Speight}
\email{G.Speight@Warwick.ac.uk}
\thanks{This work was done while the author was a PhD student of David Preiss and supported by EPSRC funding. I also thank Jaroslav Ti\v{s}er for suggesting improvements to the presentation.}

\begin{abstract}
Let $X$ be a Banach space and $2<n<\dim X$. We show there exists a directionally porous set $P$ in $X$ for which the set of $C^{1}$ surfaces of dimension $n$ meeting $P$ in positive measure is not meager. If $X$ is separable this leads to a decomposition of $X$ into the union of a $\sigma$-directionally porous set and a set which is null on residually many $C^{1}$ surfaces of dimension $n$. This is of interest in the study of $\Gamma_{n}$-null and $\Gamma$-null sets and their applications to differentiability of Lipschitz functions.
\end{abstract}
\maketitle

\section{Introduction}

We investigate the extent to which $C^{1}$ surfaces meet (directionally) porous sets (Definition \ref{porous}) in positive measure. By definition each point in a porous set sees nearby holes in the set of size proportional to their distance away. It is thus intuitively clear that porous sets are somehow small. It follows easily from the definition that porous sets are nowhere dense and hence $\sigma$-porous sets are meager. Further, if $P$ is a porous (directionally porous) set in a Banach space then the Lipschitz map $x\mapsto \mathrm{dist}(x,P)$ is Fr\'{e}chet (G\^{a}teaux) differentiable at no point of $P$. If the Banach space is finite dimensional it follows by the classical Rademacher theorem that $P$ has Lebesgue measure zero.

Porous sets have been widely studied (see \cite{Zaj1} and \cite{Zaj2} for surveys of the area) and recently have been used in the study of differentiability. It has been of much interest to what extent an analogue of Rademacher's theorem, either with G\^{a}teaux or Fr\'{e}chet differentiability, holds for Lipschitz functions defined on infinite dimensional Banach spaces (see \cite{GNFA} for an introduction and \cite{LPT}, \cite{Gamman} for some recent developments relevant to us). 

Since Lebesgue measure is unavailable in infinite dimensional Banach spaces some other notion of null set is needed to say the set of points of non differentiability of a Lipschitz function is small. The classes of most interest to us are the $\Gamma_{n}$-null sets (Definition \ref{gamman}) and $\Gamma$-null sets introduced in \cite{LPT} and \cite{Gamman} respectively. These are sets which are null on typical $n$ dimensional $C^{1}$ surfaces (or suitably defined infinite dimensional surfaces for the $\Gamma$-null case).

In separable Banach spaces there are several well known classes of null sets (for example, Aronszajn null, Haar null and Gauss null sets) with respect to which real valued Lipschitz functions are G\^{a}teaux differentiable almost everywhere (Theorem 6.42 \cite{GNFA}, also see \cite{Aron}, \cite{Christensen} and \cite{Mank}). The $\Gamma$-null sets also have this property (Theorem 2.5 \cite{Gamman}). These classes of null sets form $\sigma$-ideals (in other words they are closed under taking subsets and countable unions) so give a reasonable notion of null set.

Since $x\mapsto \mathrm{dist}(x,P)$ is G\^{a}teaux differentiable at no point of a directionally porous set $P$ it follows $\sigma$-directionally porous sets must be null in all of the above senses. A Borel set $E\subset X$ is Haar null if there exists a Borel probability measure $\mu$ on $X$ such that $\mu(x+E)=0$ for all $x \in X$. It follows from this definition that any $\sigma$-directionally porous set is null on many lines in $X$. On the other hand, it can be shown (Theorem 6.39 \cite{GNFA}, see also \cite{unexpected}) that any infinite dimensional separable Banach space contains a $\sigma$-porous set whose complement is null on all lines. Such a $\sigma$-porous set is not Haar null. This illustrates that in infinite dimensional spaces directionally porous sets are much smaller than porous sets. In finite dimensions, as one might expect, it follows from compactness of the unit sphere that porous and directionally porous sets coincide.

Proving existence of Fr\'{e}chet derivatives is much more difficult. The main known result is if a Banach space $X$ has separable dual then every real valued Lipschitz function on $X$ is Fr\'{e}chet differentiable on a dense set \cite{Frechet}. However, this is not an `almost everywhere' type result. Indeed, it is not even known if three real valued Lipschitz functions on a separable Hilbert space have a common point of Fr\'{e}chet differentiability \cite{LPT}. 

It has recently been shown if $X^{\ast}$ is separable then every real valued Lipschitz map on $X$ is Fr\'{e}chet differentiable outside a $\Gamma$-null set if and only if every $\sigma$-porous subset of $X$ is $\Gamma$-null (Corollary 3.12 \cite{Gamman}). This implies, for example, that any real valued Lipschitz function on $c_{0}$ or $C(K)$ (for $K$ countable compact) is Fr\'{e}chet differentiable outside a $\Gamma$-null set (Theorem 4.6 \cite{Gamman}). It is known (Theorem 5.4.2 \cite{LPT}) that $G_{\delta}$ sets are $\Gamma$-null if they are $\Gamma_{n}$-null for infinitely many $n$. Thus it is desirable to understand when porous sets are either $\Gamma_{n}$-null or $\Gamma$-null.

It can be shown if $n \geq \dim X$ then $\Gamma_{n}$-null and Lebesgue null sets coincide (Theorem 5.3.8 \cite{LPT}). If $n<\dim X$ the situation is much more interesting. Every $\sigma$-porous subset of a Banach space with separable dual is $\Gamma_{1}$-null (Theorem 10.4.1 \cite{LPT}) and every $\sigma$-directionally porous subset of a separable Banach space is $\Gamma_{1}$-null and $\Gamma_{2}$-null (Theorem 10.4.2 \cite{LPT}). We show the situation is very different for higher dimensional surfaces - even directionally porous sets need not be $\Gamma_{n}$ null when $2<n<\dim X$ (Theorem \ref{dp}) and if $X$ is separable the complement of a $\sigma$-directionally porous set may be $\Gamma_{n}$-null (Theorem \ref{sdp}).

One might ask if there is any differentiability result using only the notion of $\Gamma_{n}$-null sets rather than $\Gamma$-null sets. It has been shown (Theorem 11.3.6 \cite{LPT}) that if $X^{\ast}$ is separable and every porous set in $X$ can be decomposed into the union of a $\sigma$-directionally porous set and a $\Gamma_{n}$-null set of class $G_{\delta}$, then Lipschitz maps on $X$ into Banach spaces of dimension not exceeding $n$ have points of $\varepsilon$-Fr\'{e}chet differentiability for every $\varepsilon > 0$. At the time it was not yet known if porous sets in some infinite dimensional spaces were necessarily $\Gamma_{n}$-null. By showing even directionally porous sets in infinite dimensional Banach spaces need not be $\Gamma_{n}$-null we answer a question posed in \cite{LPT} (pages 186 and 203) and show that for the theorem mentioned to be meaningful the more complicated hypothesis is necessary. That is, it could not instead be more simply assumed porous sets are $\Gamma_{n}$-null.

We now give the formal definitions that will be relevant for us. In what follows $B(x,r)$ will denote the open ball in $X$ with centre $x \in X$ and radius $r>0$.

\begin{definition}\label{porous}
A set $P \subset X$ is called porous if there exists $0<\rho<1$ such that for all $x \in P$ and $\delta >0$ there exists $y \in X$ with $\|y-x\|<\delta$ such that
\[B(y,\rho \|y-x\|)\cap P = \varnothing.\]
We refer to the ball $B(y,\rho \|y-x\|)$ as a hole in $P$.

A set $P \subset X$ is called directionally porous if there exists $0<\rho<1$ such that for all $x \in P$ there exists $v \in X$ with $\|v\|=1$ such that for any $\delta>0$ there exists $t \in \mathbb{R}$ with $|t|<\delta$ such that
\[B(x+tv,\rho |t|)\cap P = \varnothing.\]

We refer to the constant $\rho$ appearing in the above definitions as a porosity constant of $P$. 

A set is called $\sigma$-porous ($\sigma$-directionally porous) if it is a union of countably many porous (directionally porous) sets.
\end{definition}

Recall a subset of a metric space is called typical, or residual, if its complement is meager.

\begin{definition}\label{gamman}
A $C^{1}$ surface of dimension $n$ in $X$ is a $C^{1}$ map $f\colon [0,1]^{n}\to X$. We define the $C^{1}$ norm by
\[ \| f \|_{C^{1}} = \max (\| f \|_{\infty}, \| {\partial f}/{\partial x_{1}}\|_{\infty}, \ldots, \| {\partial f}/{\partial x_{n}}\|_{\infty}) \]
where $\|\cdot \|_{\infty}$ denotes the supremum norm. Denote the space of $C^{1}$ surfaces of dimension $n$ in $X$ with $C^{1}$ norm by $\Gamma_{n}(X)$.

A set $N\subset X$ is called $\Gamma_{n}$-null if
\[ \mathcal{H}^{n}(f([0,1]^{n})\cap N)=0\]
for residually many $f \in \Gamma_{n}(X)$.
\end{definition}

\begin{remark}
In \cite{LPT} a set $N\subset X$ is defined to be $\Gamma_{n}$-null if
\[ \mathcal{L}^{n}\{t \in [0,1]^{n}: f(t) \in N\}=0\]
for residually many $f \in \Gamma_{n}(X)$. It is easy to see a set which is $\Gamma_{n}$-null in this sense is also $\Gamma_{n}$-null in the sense of Definition \ref{gamman}. In fact the two definitions are equivalent for $n\leq \dim X$. This follows from Lemma 5.3.5 \cite{LPT} which states that for a typical surface $f \in \Gamma_{n}(X)$ the derivative $df$ has rank equal to $\min(n,\dim X)$ at Lebesgue almost every point of $[0,1]^{n}$.
\end{remark}

We can now state precisely the results of this paper.

\begin{theorem}\label{dp}
Let $2<n<\dim X$. Then there exists a directionally porous set $P\subset X$ which is not $\Gamma_{n}$-null.
\end{theorem}

\begin{theorem}\label{sdp}
Let $X$ be separable and $2<n<\dim X$. Then there exists a $\sigma$-directionally porous set $Q\subset X$ and a $\Gamma_{n}$-null set $N\subset X$ such that $X=Q\cup N$.
\end{theorem}

Intuitively the reason for the difference between the cases $n\leq 2$ and $n>2$ is that for $n>2$ modifying a surface to go through a nearby hole causes an area change comparable to the size of the hole (this is stated precisely in Proposition \ref{area} which is a corollary of a Poincar\'e type inequality). By a careful construction we exploit this fact to construct a porous set where the size of holes a surface meets is controlled by the area of the surface.

Fix $n>2$ throughout the remainder of the paper. We focus on proving Theorem \ref{dp} in the case $X=\mathbb{R}^{n+1}$. Since porous sets and directionally porous sets coincide in finite dimensions it suffices to construct a porous set in $\mathbb{R}^{n+1}$ which is not $\Gamma_{n}$-null. In the final section we deduce the result for a general Banach space and see Theorem \ref{sdp} follows.

\section{Geometric conditions for the construction}

Let $A$ be a subset of a complete metric space $M$ and $U_{k}$ be a sequence of open sets which are each dense in a fixed open set $U\subset M$. If $\bigcap_{k=1}^{\infty} U_{k} \subset A$ then $U\setminus A$ is meager in $M$. Hence by the Baire Category Theorem $A$ is not meager in $M$. We now investigate what this means in terms of surfaces and porous sets.

Define the $C^{1}$ surface $p\colon [0,1]^{n} \to \mathbb{R}^{n+1}$ by $p(x)=(x,0)$ and and let $c=(1/2,\ldots,1/2) \in \mathbb{R}^{n}$. Fix $0<r<1/32$ to be chosen small later. 

If $g\colon B(c,t) \to \mathbb{R}$ let
\[G(g)=\{(x,g(x)):x \in B(c,t)\}\]
denote the graph of $g$. 

It will usually be simpler to work with surfaces represented as graphs.

\begin{lemma}
There exists $s>0$ (fixed and independent of $r$) and $\delta(r)>0$ such that for all $f \in B_{C^{1}}(p,\delta(r))$ there exists $g\colon B(c,s) \to \mathbb{R}$ of class $C^{1}$ with $\|g\|_{C^{1}}\leq r$ and
\[f([0,1]^{n})\cap (B(c,s) \times \mathbb{R}) = G(g).\]
\end{lemma}

\begin{proof}
Let $e_{1}, \ldots, e_{n}$ denote the standard basis of $\mathbb{R}^{n}$. We apply the Inverse Function Theorem to $\widetilde{f}=(f_{1},\ldots, f_{n})$ which is a $C^{1}$ mapping from $[0,1]^{n}$ to $\mathbb{R}^{n}$ and satisfies $\| \partial \widetilde{f} / \partial x_{i} - e_{i} \|_{\infty}<\delta(r)$ for $i=1,\ldots, n$. Note the size of the neighbourhood on which $\widetilde{f}$ is invertible can be made independent of the particular function $f$ and depend only on bounds on its derivative (this is clear from the proof of the Inverse Function Theorem given in \cite{IFT}). 

Provided $\delta(r)$ is sufficiently small (independently of $f$) $\widetilde{f}$ is invertible on a region whose image contains $c$ and has a $C^{1}$ inverse $(\widetilde{f})^{-1}\colon B(c,s)\to \mathbb{R}^{n}$ with $s>0$ independent of $r, f$ and $\|\partial (\widetilde{f})^{-1} / \partial x_{i}\|_{\infty} \leq 2$ for $i=1, \ldots, n$. Define the $C^{1}$ map $g\colon B(c,s)\to \mathbb{R}$ by $g=f_{n+1}\circ (\widetilde{f})^{-1}$. Since $\|f_{n+1}\|_{C^{1}} <\delta(r)$ we have $\|g\|_{C^{1}}\leq r$ for sufficiently small $\delta (r)$.
\end{proof}

Let $\mathcal{F}$ be a (not necessarily disjoint) collection of open balls in $\mathbb{R}^{n+1}$ and $L>1$. Intuitively $\mathcal{F}$ will correspond to the holes of a porous set. Define, for $k \geq 1$,
\begin{equation}\label{Pk}
P_{k}= \bigcup_{\substack{B \in \mathcal{F}\\ \mathrm{diam} B < 1/k}} LB, \quad H=\bigcup_{B \in \mathcal{F}} B,
\end{equation}
where $LB$ denotes the ball with the same centre as $B$ but radius enlarged by factor ~$L$.

Notice $P_{k}$ is open and $P_{k+1}\subset P_{k}$ for all $k \geq 1$. Further
\[P= \left( \bigcap_{k=1}^{\infty} P_{k} \right) \setminus H\]
is porous with porosity constant $1/L$. Informally $\bigcap_{k=1}^{\infty} P_{k}$ consists of points which have a nearby hole, of radius proportional to its distance from the point, of diameter less than $1/k$ and $H$ is the union of all the holes (of any size).

Given $A\subset \mathbb{R}^{n+1}$ denote the set of $n$ dimensional surfaces which meet $A$ in $n$ dimensional measure greater than $\alpha \geq 0$ by
\[S_{n}(A,\alpha)=\{f \in \Gamma_{n}(X):\mathcal{H}^{n}(A\cap f([0,1]^{n}))>\alpha\}.\]

To prove Theorem \ref{dp} in the case $X=\mathbb{R}^{n+1}$ it suffices to construct $\mathcal{F}$ and $L>1$ for which there is $\alpha>0$ such that, for sufficiently small $r$,
\begin{equation}\label{first}
S_{n}((B(c,s)\times \mathbb{R})\cap P_{k},\alpha) \text{ is dense in } B_{C^{1}}(p,\delta(r)) \text{ for all } k \geq 1
\end{equation}
and
\begin{equation}\label{second}
\mathcal{H}^{n}(f([0,1]^{n})\cap (B(c,s)\times \mathbb{R}) \cap H) \leq \alpha/4 \text{ for all } f \in B_{C^{1}}(p,\delta(r)).
\end{equation} 

As the sets $P_{k}$ are decreasing it follows
\[ \bigcap_{k=1}^{\infty} S_{n}((B(c,s)\times \mathbb{R})\cap P_{k},\alpha) \subset S_{n}\left( \bigcap_{k=1}^{\infty} (B(c,s)\times \mathbb{R})\cap P_{k},\alpha/2 \right).\] 
Since the sets $P_{k}$ are open each of the sets $S_{n}((B(c,s)\times \mathbb{R})\cap P_{k},\alpha)$ is also open. By \eqref{first} they are also dense in $B_{C^{1}}(p,\delta(r))$ so it follows, by the discussion at the start of this section, that the complement of
\[S_{n}\left( \bigcap_{k=1}^{\infty} (B(c,s)\times \mathbb{R})\cap P_{k},\alpha/2 \right)\]
is meager in $B_{C^{1}}(p,\delta(r))$. Then by \eqref{second},
\[S_{n}\left( \bigcap_{k=1}^{\infty} (B(c,s)\times \mathbb{R})\cap P_{k},\alpha/2 \right) \subset S_{n}(P,\alpha/4)\]
so the complement of $S_{n}(P,\alpha/4)$ is meager in $B_{C^{1}}(p,\delta(r))$. By the Baire Category Theorem, as in the discussion at the start of this section, this will prove Theorem \ref{dp} in the case $\mathbb{R}^{n+1}$.

Thus proving Theorem \ref{dp} amounts to constructing smaller and smaller open balls (intuitively the holes of the porous set) whose enlargements mostly cover surfaces in a countable dense set and so that the intersection of all balls with any one surface is kept small.

\section{Construction of the porous set}

One of the requirements on $\mathcal{F}$ is that the enlargements of balls in $\mathcal{F}$ of arbitrarily small radii must cover a fixed proportion of each surface in a countable dense subset of $B_{C^{1}}(p,\delta(r))$. To do this it is natural to choose our countable dense subset to be as simple as possible. 

The following lemma follows from Corollary 10.2.2 \cite{LPT}, which states surfaces in which almost every point has a neighbourhood on which the surface is affine are dense in $\Gamma_{n}(\mathbb{R}^{n+1})$, and allows us to choose a countable dense subset consisting of surfaces that are mostly covered by countably many planes. 

Let $\omega_{n}$ denote the volume of a unit ball in $\mathbb{R}^{n}$.

\begin{lemma}
There is a countable dense set of surfaces $\{f_{l}\}_{l=1}^{\infty}$ in $B_{C^{1}}(p,\delta(r))$ and a countable collection of affine planes $A_{l}(x)=(x,a_{l}(x))$, where $x \in [0,1]^{n}$, $\nabla a_{l}$ is constant and $|\nabla a_{l}| \leq r$, such that
\[\mathcal{H}^{n} \left( \bigcup_{l=1}^{\infty} f_{l}([0,1]^{n}) \setminus \bigcup_{l=1}^{\infty} A_{l}([0,1]^{n}) \right) < \omega_{n}s^{n}/4.\]
Without loss of generality let $f_{1}=A_{1}=p$. 
\end{lemma}

We will define $\mathcal{F}$ and $L>1$ so that, if $P_{k}$ is defined as in \eqref{Pk},
\begin{equation}\label{third}
\mathcal{H}^{n}(A_{l}([0,1]^{n})\cap (B(c,s)\times \mathbb{R}) \setminus P_{k}) \leq \omega_{n}s^{n}/2^{l+2} \text{ for all } k,l\geq 1. 
\end{equation}
This will imply
\[\mathcal{H}^{n}(f_{l}([0,1]^{n})\cap (B(c,s)\times \mathbb{R})\cap P_{k}) > \omega_{n}s^{n}/2 \text{ for all } k,l \geq 1\]
and hence $f_{l} \in S_{n}((B(c,s)\times \mathbb{R})\cap P_{k},\omega_{n}s^{n}/2)$ for all $k,l \geq 1$ which implies \eqref{first} with $\alpha = \omega_{n}s^{n}/2$.

In Section 5 we then show that if $r$ is sufficiently small our construction ensures
\[\mathcal{H}^{n}(f([0,1]^{n})\cap (B(c,s)\times \mathbb{R}) \cap H) \leq \omega_{n}s^{n}/8 \text{ for all } f \in B_{C^{1}}(p,\delta(r))\]
which is \eqref{second} with $\alpha = \omega_{n}s^{n}/2$. Theorem \ref{dp} hence follows in the case $X=\mathbb{R}^{n+1}$.

Choose a sequence $m_{k}$ of natural numbers with $m_{1}=1$ in which every natural number is repeated infinitely many times. Fix a sequence $\varepsilon_{i}>0$ with $\varepsilon_{i}<1/2^{i}$ and $3\sum_{i=1}^{\infty}\varepsilon_{i}\leq 1/64$ to be chosen as small as required later.

Let $r_{0}=s$. For $k\geq 1$ we will inductively define $\mathcal{F}_{k}$ and $r_{k}>0$ such that $\mathcal{F}_{k}$ consists of finitely many balls in $\mathbb{R}^{n+1}$, of radius less then $\varepsilon_{k}^{3} r_{k-1}$ and greater than $r_{k}$, whose enlargements (by a factor independent of $k$) mostly cover $A_{m_{k}}([0,1]^{n})$. The family $\mathcal{F}_{k}$ will consist of subfamilies of balls on different levels relative to $A_{m_{k}}$ constructed so there are relatively few balls on any one level. Each of these subfamilies of $\mathcal{F}_{k}$ is formed by lifting families of balls $\mathcal{G}_{k}^{l}$ in $\mathbb{R}^{n}$ to $\mathbb{R}^{n+1}$.

Fix $k\geq 1$ for which $r_{k-1}$ has been defined. We show how to define $\mathcal{G}_{k}^{l}$ inductively and hence define $\mathcal{F}_{k}$ and $r_{k}$. Let $\mathcal{G}_{k}^{0}=\varnothing$ and $r_{k}^{0}=r_{k-1}$.

\begin{lemma}\label{inductive}
Fix $l\geq 0$. Suppose families $\mathcal{G}_{k}^{0}, \ldots, \mathcal{G}_{k}^{l}$ of finitely many balls in $\mathbb{R}^{n}$ and $r_{k}^{0}, \ldots, r_{k}^{l}>0$ have been constructed so that (if $l \geq 1$):
\begin{itemize}
	\item For each $1 \leq q \leq l$, $\mathcal{G}_{k}^{q}$ consists of balls of radius $r_{k}^{q} \leq \varepsilon_{k}^{3}r_{k}^{q-1}$.
	\item For each $1 \leq q \leq l$, the balls in
	\[\{(1/\varepsilon_{k}^{3})B:B \in \mathcal{G}_{k}^{q}\}\]
	are disjoint and contained inside $B(c,s)$.
	\item Any two balls from
	\[\{(1/\varepsilon_{k}^{3})B:B \in \mathcal{G}_{k}^{1} \cup \ldots \cup \mathcal{G}_{k}^{l}\}\]
	are either disjoint or one is contained inside the other.
\end{itemize}
Then there exists a family $\mathcal{G}_{k}^{l+1}$ of finitely many balls in $\mathbb{R}^{n}$ which cover at least $\varepsilon_{k}^{3n}/2^{n+1}$ proportion of
\[B(c,s) \setminus \bigcup_{B \in \mathcal{G}_{k}^{1} \cup \ldots \cup \mathcal{G}_{k}^{l}} B\]
and $0<r_{k}^{l+1}\leq \varepsilon_{k}^{3}r_{k}^{l}$ such that:
\begin{itemize}
	\item Balls in $\mathcal{G}_{k}^{l+1}$ have radius $r_{k}^{l+1}$.
	\item The balls in
	\[\{(1/\varepsilon_{k}^{3})B:B \in \mathcal{G}_{k}^{l+1}\}\]
	are disjoint and contained inside $B(c,s)$.
	\item Any two balls from 
	\[\{(1/\varepsilon_{k}^{3})B:B \in \mathcal{G}_{k}^{1} \cup \ldots \cup \mathcal{G}_{k}^{l+1}\}\]
	are either disjoint or one is contained inside the other.
\end{itemize}
\end{lemma}

\begin{proof}
Find $0<r_{k}^{l+1}\leq \varepsilon_{k}^{3}r_{k}^{l}$ such that at least half of the points in 
\[B(c,s) \setminus \bigcup_{B \in \mathcal{G}_{k}^{1} \cup \ldots \cup \mathcal{G}_{k}^{l}} B\]
are of distance at least $r_{k}^{l+1}/\varepsilon_{k}^{3}$ away from the closed set
\[(\mathbb{R}^{n}\setminus B(c,s)) \cup \bigcup_{B \in \mathcal{G}_{k}^{1}\cup \ldots \cup \mathcal{G}_{k}^{l}} \partial ((1/\varepsilon_{k}^{3})B).\]
By inductively choosing balls of radius $r_{k}^{l+1}$, whose centres are of distance at least $2r_{k}^{l+1}/\varepsilon_{k}^{3}$ from centres of previously chosen balls, we can find a finite family $\mathcal{G}_{k}^{l+1}$ of balls of radius $r_{k}^{l+1}$ such that balls in $\{(1/\varepsilon_{k}^{3})B:B \in \mathcal{G}_{k}^{l+1}\}$ are disjoint, contained in
\[B(c,s) \setminus \bigcup_{B \in \mathcal{G}_{k}^{1}\cup \ldots \cup \mathcal{G}_{k}^{l}} \partial ((1/\varepsilon_{k}^{3})B)\]
and balls in $\{(2/\varepsilon_{k}^{3})B:B \in \mathcal{G}_{k}^{l+1}\}$ cover at least half of
\[B(c,s)\setminus \bigcup_{B \in \mathcal{G}_{k}^{1}\cup \ldots \cup \mathcal{G}_{k}^{l}} B.\]
Then
\begin{align*}
\mathcal{L}^{n}\left( \bigcup_{B \in \mathcal{G}_{k}^{l+1}} B \right) &= \sum_{B \in \mathcal{G}_{k}^{l+1}} \mathcal{L}^{n}(B)\\
&=(\varepsilon_{k}^{3}/2)^{n}\sum_{B \in \mathcal{G}_{k}^{l+1}} \mathcal{L}^{n}((2/\varepsilon_{k}^{3})B)\\
&\geq (\varepsilon_{k}^{3}/2)^{n}\mathcal{L}^{n}\left( B(c,s)\setminus \bigcup_{B \in \mathcal{G}_{k}^{1}\cup \ldots \cup \mathcal{G}_{k}^{l}} B \right)/2
\end{align*}
as required.
\end{proof}

Once $\mathcal{G}_{k}^{l}$ has been defined if
\[\mathcal{L}^{n} \left( B(c,s)\setminus \bigcup_{B \in \mathcal{G}_{k}^{1}\cup \ldots \cup \mathcal{G}_{k}^{l}} B \right) \leq \omega_{n}s^{n}/2^{k+3}\]
we stop, set $r_{k}=r_{k}^{l}$ and let
\[\mathcal{G}_{k}=\mathcal{G}_{k}^{1}\cup \ldots \cup \mathcal{G}_{k}^{l}.\]
Otherwise we continue and define $\mathcal{G}_{k}^{l+1}$ and $r_{k}^{l+1}$ as in Lemma \ref{inductive}. Since at each stage we cover a fixed proportion (independent of $l$) of the region not yet covered this process stops in a finite number of steps.

To define $\mathcal{F}_{k}$ from $\mathcal{G}_{k}$ we replace each ball
\[B(x,t)\in \mathcal{G}_{k}\]
by
\[B((x,a_{m_{k}}(x)+2t),t) \in \mathcal{F}_{k}.\]

Inductively define $\mathcal{F}_{k}$ and $r_{k}$ as above for all $k \geq 1$, then let $\mathcal{F}=\cup_{k=1}^{\infty} \mathcal{F}_{k}$ and $L=\sqrt{10}$.

\begin{proposition}\label{cover}
If $P_{k}$ are defined from $\mathcal{F}, L$ as in \eqref{Pk} then
\[\mathcal{H}^{n}(A_{m_{k}}([0,1]^{n})\cap (B(c,s)\times \mathbb{R}) \setminus P_{k}) \leq \omega_{n}s^{n}/2^{k+2}\]
for all $k \geq 1.$
\end{proposition}

\begin{proof}
Notice if $B(x,t) \in \mathcal{G}_{k}$ then, since $|\nabla a_{m_{k}}|\leq 1$,
\[A_{m_{k}}([0,1]^{n})\cap (B(x,t) \times \mathbb{R}) \subset B((x,a_{m_{k}}(x)+2t),t\sqrt{10}) \subset P_{k}.\]
Hence
\begin{align*}
\mathcal{H}^{n}(A_{m_{k}}([0,1]^{n})\cap (B(c,s)\times \mathbb{R}) \setminus P_{k}) &\leq \mathcal{H}^{n}\left( A_{m_{k}}([0,1]^{n})\cap \left(B(c,s)\setminus \bigcup_{B \in \mathcal{G}_{k}} B \right) \times \mathbb{R} \right)\\
&\leq 2\mathcal{L}^{n} \left( B(c,s)\setminus \bigcup_{B \in \mathcal{G}_{k}} B \right) \\
&\leq \omega_{n}s^{n}/2^{k+2}.
\end{align*}
\end{proof}

Since the sequence $m_{k}$ takes the value of every natural number infinitely often and the sets $P_{k}$ are decreasing, Proposition \ref{cover} implies \eqref{third}.

\section{Area estimates for surfaces}

It remains to show the intersection of balls in $\mathcal{F}$ with any fixed surface in $B_{C^{1}}(p,\delta(r))$ is small. 

We now establish several results which later allow us to show how passing through many holes forces the area of a surface $f$ to increase and to distinguish between area increments arising from holes of different sizes.

The following lemma is an adaptation of Theorem 1(iii) in Section 5.6.1 of \cite{EG}. The proof is essentially the same as in \cite{EG} but with the numbers $1$ and $n/(n-1)$ replaced by $2$ and its corresponding Sobolev conjugate $2n/(n-2)$. Notice the assumption $n>2$ is essential. 

In what follows $C$ will be a positive constant (depending only on $n$) whose value may be different in different expressions. We will sometimes write $2C$ instead of $C$ if it is important but not obvious an extra constant term has been added.

\begin{lemma}\label{sobolev}
For each $0<\alpha\leq 1$ there exists a constant $C(\alpha)>0$ such that
\[\|g\|_{L^{2n/(n-2)}(B)}\leq C(\alpha) \|\nabla g\|_{L^{2}(B)}\]
for all balls $B\subset \mathbb{R}^{n}$ and $g \in W^{n,2}_{\mathrm{loc}}(\mathbb{R}^{n})$ satisfying
\[ \frac{\mathcal{L}^{n}(B\cap \{g=0\})}{\mathcal{L}^{n}(B)}\geq \alpha.\]
\end{lemma}

\begin{proof}
Suppose $B=B(x,t)\subset \mathbb{R}^{n}$ and $g \in W^{n,2}_{\mathrm{loc}}(\mathbb{R}^{n})$ satisfies
\[ \frac{\mathcal{L}^{n}(B\cap \{g=0\})}{\mathcal{L}^{n}(B)}\geq \alpha.\]
By the triangle inequality
\[\|g\|_{L^{2n/(n-2)}(B)} \leq \|g-(g)_{B}\|_{L^{2n/(n-2)}(B)}+|(g)_{B}|(\mathcal{L}^{n}(B))^{(n-2)/2n}\]
where
\[(g)_{B}=\frac{1}{\mathcal{L}^{n}(B)}\int_{B} g.\]
We estimate the two terms individually. 

Since $n>2$, Poincar\'e's inequality (Theorem 2 of Section 4.5.2 \cite{EG}) states,
\[ \left( \dashint_{B} |g-(g)_{B}|^{2n/(n-2)} \right)^{(n-2)/2n} \leq Ct \left(\dashint_{B} |\nabla g|^{2} \right)^{1/2}.\]
As $\mathcal{L}^{n}(B)=\omega_{n}t^{n}$ this implies
\[ \|g-(g)_{B}\|_{L^{2n/(n-2)}(B)}\leq C\| \nabla g\|_{L^{2}(B)}.\]

Next notice
\[ |(g)_{B}|(\mathcal{L}^{n}(B))^{(n-2)/2n}\leq (\mathcal{L}^{n}(B))^{-(n+2)/2n}\int_{B} |g|\chi_{\{g\neq 0\}}\]
where $\chi_{A}$ denotes the characteristic function of a set $A$. Applying H\"{o}lder's inequality with exponents $2n/(n-2)$ and $2n/(n+2)$ implies
\[ |(g)_{B}|(\mathcal{L}^{n}(B))^{(n-2)/2n}\leq (\mathcal{L}^{n}(B))^{-(n+2)/2n} \|g\|_{L^{2n/(n-2)}(B)}(\mathcal{L}^{n}(B\cap \{g\neq 0\}))^{(n+2)/2n}.\]
The assumption on $g$ then yields
\[ |(g)_{B}|(\mathcal{L}^{n}(B))^{(n-2)/2n}\leq (1-\alpha)^{(n+2)/2n} \|g\|_{L^{2n/(n-2)}(B)}.\]

Putting the two estimates together implies
\[ \|g\|_{L^{2n/(n-2)}(B)} \leq C\|\nabla g\|_{L^{2}(B)} + (1-\alpha)^{(n+2)/2n}\|g\|_{L^{2n/(n-2)}(B)}.\]
Rearranging this expression and relabelling constants leads to the desired inequality.
\end{proof}

Our use of Lemma \ref{sobolev} is expressed in the following proposition. The underlying idea will be that if a surface passes through different vertical levels it creates an area change comparable to holes at those heights. 

\begin{proposition}\label{area}
Suppose $h>0$ and $B \subset \mathbb{R}^{n}$ is a ball of radius at least $h$. Then
\[\omega_{n}h^{n} \leq C\int_{B \cap \{g\geq h/2\}} |\nabla g|^{2}\]
for all $g\colon B \to \mathbb{R}$ of class $C^1$ with $|\nabla g|\leq 1$ such that
\[ \frac{\mathcal{L}^{n}(B\cap \{g\leq h/2\})}{\mathcal{L}^{n}(B)}\geq 1/2\]
and $g(\widetilde{x})\geq h$ for some $\widetilde{x} \in B$.
\end{proposition}

\begin{proof}
Let $g\colon B \to \mathbb{R}$ be as in the statement of the proposition. Extend $g$ to a function in $W^{n,2}_{\mathrm{loc}}(\mathbb{R}^{n})$ and let $(g-h/2)^{+}$ denote the function which equals $g-h/2$ if $g \geq h/2$ and is zero otherwise. Then $(g-h/2)^{+}$ belongs to $W^{n,2}_{\mathrm{loc}}(\mathbb{R}^{n})$ and satisfies
\[ \frac{\mathcal{L}^{n}(B\cap \{(g-h/2)^{+}=0\})}{\mathcal{L}^{n}(B)}\geq 1/2.\]
Hence by Lemma \ref{sobolev}
\[\|(g-h/2)^{+}\|_{L^{2n/(n-2)}(B)}\leq C(1/2) \|\nabla (g-h/2)^{+}\|_{L^{2}(B)}.\]
Since $g$ is $C^{1}$, $|\nabla g| \leq 1$ in $B$, $g(\widetilde{x})\geq h$ and $B$ has radius at least $h$ it follows there is $C>0$ such that $g\geq 3h/4$ inside a region of measure at least $Ch^{n}$ and contained in $B$. Hence
\[((h/4)^{2n/(n-2)}Ch^{n})^{(n-2)/2n} \leq C(1/2)  \left( \int_{B \cap \{g\geq h/2\}} |\nabla g|^{2}\right)^{1/2}.\]
Simplifying this expression and relabelling the constants gives the claimed inequality.
\end{proof}

Our holes were constructed relative to planes with varying directions. The following proposition will later allow us to use the previous result in this context.

\begin{proposition}\label{flatten}
Let $B=B(x,t)\subset \mathbb{R}^{n}$, $g\colon B \to \mathbb{R}$ of class $C^{1}$ with $|\nabla g|\leq 1$ and $a\colon B \to \mathbb{R}$ of class $C^{1}$ with $\nabla a$ constant and $|\nabla a|\leq 1$. Suppose $\varepsilon>0$ and $|g-a|\leq \varepsilon t$ in $B$. Then
\[\left| \int_{B}(|\nabla g|^{2} - |\nabla a|^{2}- |\nabla (g-a)|^{2})\right| \leq C \varepsilon \mathcal{L}^{n}(B).\]
\end{proposition}

\begin{proof}
Write $g=(g-a)+a$ and expand the terms to obtain
\begin{align*}
\left| \int_{B}(|\nabla g|^{2} - |\nabla a|^{2}- |\nabla (g-a)|^{2})\right| & = 2 \left| \int_{B} \sum_{i=1}^{n} \frac{\partial a}{\partial x_{i}} \frac{\partial (g-a)}{\partial x_{i}} \right| \\
& \leq C \sum_{i=1}^{n} \left| \int_{B} \frac{\partial (g-a)}{\partial x_{i}} \right|\\
& \leq C\varepsilon \mathcal{L}^{n}(B).
\end{align*}
using the Divergence Theorem and the assumption $|g-a|\leq \varepsilon t$.
\end{proof}

In what follows we will need to smooth our surfaces in order to distinguish between oscillations arising from different sized holes. The following definition recalls some basic notions \cite{EG}.

\begin{definition}
Define the $C^{\infty}$ function $\eta \colon \mathbb{R}^{n} \to \mathbb{R}$ by $\eta(x)=\kappa \exp (1/(|x|^{2}-1))$ for $|x|<1$ and $\eta(x)=0$ for $|x|\geq 1$, where $\kappa > 0$ is chosen so that
\[\int_{\mathbb{R}^{n}} \eta(x)\, \dd x = 1.\]
For $\varepsilon > 0$ define the standard mollifier $\eta_{\varepsilon} \colon \mathbb{R}^{n} \to \mathbb{R}$ by $\eta_{\varepsilon}(x)=(1/\varepsilon^{n})\eta(x/\varepsilon)$. If $g \in L^1_{\mathrm{loc}}(U)$ and $\varepsilon>0$ define
\[g^{\varepsilon}(x)=\int_{B(x,\varepsilon)} \eta_{\varepsilon}(x-y)g(y)\dd y\]
for $x \in U_{\varepsilon}=\{x \in U:\mathrm{dist} (x,\partial U)>\varepsilon\}$.
\end{definition}

We informally refer to $g^{\varepsilon}$ as a smoothing of $g$. The following facts are either well known \cite{EG} or easy to prove using the definition.

\begin{lemma}
Let $0<\varepsilon<t$ and $g\colon B(x,t) \to \mathbb{R}$ be of class $C^{1}$ and satisfy $|\nabla g|\leq 1 $ in $B(x,t)$. Then $g^{\varepsilon}$ is of class $C^{\infty}(B(x,t-\varepsilon))$ and $|g^{\varepsilon}-g|\leq \varepsilon$.
\end{lemma}

The following proposition gives an estimate on the smoothing of a function whose values are inside a small interval.

\begin{proposition}\label{smoothedarea}
Let $B=B(x,t)\subset \mathbb{R}^{n}$ and $0<\varepsilon < 1/2$. Suppose $g\colon B \to \mathbb{R}$ is of class $C^{1}$ with $|\nabla g| \leq 1$ and $|g|\leq \varepsilon^{2}t$ in $B$. Then
\[ \int_{B(x,t-\varepsilon t)} |\nabla (g^{\varepsilon t})|^{2} \leq C\varepsilon^{2} \mathcal{L}^{n}(B).\] 
\end{proposition}

\begin{proof}
For $x \in B(x,t-\varepsilon t)$, using the formula for the derivative of a convolution,
\begin{align*}
\left| \frac{\partial g^{\varepsilon t}}{\partial x_{i}}(x)\right| &=\left| \int_{B(x,\varepsilon t)} \frac{\partial \eta_{\varepsilon t}}{\partial x_{i}}(x-y)g(y)\dd y\right|\\
&\leq \int_{B(x,\varepsilon t)} (C/\varepsilon^{n+1} t^{n+1})\varepsilon^{2} t \dd y\\
&\leq \omega_{n} \varepsilon^{n} t^{n}(C/\varepsilon^{n+1} t^{n+1})\varepsilon^{2} t\\
&=C\varepsilon.
\end{align*}
Hence $|\nabla (g^{\varepsilon t})|\leq C\varepsilon$ in $B(x,t-\varepsilon t)$ so the desired inequality follows.
\end{proof}

We will need to smooth locally at different scales in different regions. The following proposition gives a function which allows us to interpolate smoothly between two functions.

\begin{proposition}\label{interpolate}
Let $B(x,t)\subset \mathbb{R}^{n}$ and $0<\varepsilon <1/2$. Then there exists a function $w\colon \mathbb{R}^{n} \to \mathbb{R}$ of class $C^{1}$ with $0\leq w \leq 1$ such that:
\begin{itemize}
	\item $w=1$ inside $B(x,t-2\varepsilon t)$.
	\item $w=0$ outside $B(x,t-\varepsilon t)$.
	\item $|\nabla w| \leq 3/(\varepsilon t)$ in $B(x,t-\varepsilon t)\setminus B(x,t-2\varepsilon t)$.
\end{itemize}

Suppose $g_{1}\colon B(x,t-\varepsilon t) \to \mathbb{R}$ and $g_{2}\colon B(x,t) \to \mathbb{R}$ are $C^{1}$ with bounded derivatives and $|g_{1}-g_{2}|\leq \varepsilon^{2} t$ in $B(x,t-\varepsilon t)\setminus B(x,t-2\varepsilon t)$.
 
Then the function $v: B(x,t) \to \mathbb{R}$ defined to be $wg_{1}+(1-w)g_{2}$ in $B(x,t-\varepsilon t)$ and $g_{2}$ in $B(x,t)\setminus B(x,t-\varepsilon t)$ is $C^{1}$ with 
\[|\nabla v|\leq \max(|\nabla g_{1}|, |\nabla g_{2}|)+3\varepsilon\]
in $B(x,t-\varepsilon t)$.
\end{proposition}

\begin{proof}
Define $\widetilde{w}\colon \mathbb{R}^{n} \to \mathbb{R}$ such that:
\begin{itemize}
	\item $\widetilde{w}=1$ inside $B(x,t-5\varepsilon t /3).$
	\item $\widetilde{w}=0$ outside $B(x,t-4\varepsilon t /3).$
	\item $\widetilde{w}$ interpolates between $0$ and $1$ as a linear function of distance to $x$.
\end{itemize}
It is easy to see the function $w=\widetilde{w}^{\varepsilon t/3}$ then has the properties required.

Clearly
\[\nabla (wg_{1}+(1-w)g_{2}) = w\nabla g_{1}+(1-w)\nabla g_{2}+\nabla w (g_{1}-g_{2})\]
in $B(x,t-\varepsilon t)$ so
\[|\nabla v| \leq \max(|\nabla g_{1}|, |\nabla g_{2}|)+3\varepsilon\]
in $B(x,t-\varepsilon t)$.
Since $w$ is $C^{1}$ and equal to zero outside $B(x,t-\varepsilon t)$ it also follows $v$ is $C^{1}$ on $B(x,t)$.
\end{proof}

\section{Holes are small on surfaces}

We now need to show \eqref{second} with $\alpha = w_{n}s^{n}/2$ which intuitively states that the intersection of surfaces in $B_{C^{1}}(p,\delta(r))$ with $(B(c,s)\times \mathbb{R})\cap H$ is small. To do this we control the intersection of any surface $f \in B_{C^{1}}(p,\delta(r))$, which is the graph of some function $g$ over $B(c,s)$, with $H$ using $\int_{B(c,s)} |\nabla g|^{2}$. This integral is approximately the $n$ dimensional area difference between the image of $f$ over $B(c,s)$ and the flat disc $B(c,s)\times \{0\}$. The intuition is thus that passing through holes forces the surface to oscillate in a way that increases its area.

If $x \in \mathbb{R}^{n+1}$ let $x' \in \mathbb{R}^{n}$ be the point consisting of the first $n$ coordinates of $x$. If $B \subset \mathbb{R}^{n+1}$ is a ball of radius $t>0$ let $|B|=\omega_{n}t^n$ be the area of an $n$ dimensional cross section of $B$.

\begin{proposition}\label{mainprop}
There is a constant $C\geq 1$ such that
\[ \sum_{\substack{B \in \mathcal{F}\\ G(g) \cap B \neq \varnothing}} |B| \leq C\int_{B(c,s)}|\nabla g|^{2}+C\sum_{i=1}^{\infty} \varepsilon_{i}\]
for all $g\colon B(c,s) \to \mathbb{R}$ of class $C^{1}$ satisfying $\|g\|_{C^{1}}\leq 1/64$.
\end{proposition}

\begin{proof}[Proof of Proposition \ref{mainprop}]

To prove the proposition we will reformulate it in a way that allows us to use induction (Lemma \ref{base} and Lemma \ref{step}). At times in the argument we will interpolate between pieces of the surface which are smoothed on different scales. Thus it is natural to control the intersection of surfaces with slightly enlarged holes and begin with weaker bounds on $|\nabla g|$ which we tighten after each induction step.

For $k\geq 1$ let $K_{k}=2-\sum_{i=1}^{k} (1/2)^{i}$. If $B=B(x,t)\in \mathcal{F}_{k}$ for some $k \geq 1$ denote $t'=t/\varepsilon_{k}^{3}$ and $B'=B(x',t') \subset \mathbb{R}^{n}$. Note this definition makes sense because any ball can be in $\mathcal{F}_{k}$ for at most one $k\geq 1$. Given $g\colon B(c,s) \to \mathbb{R}$ of class $C^{1}$ with $|\nabla g|\leq 1/32$ if $B \in \mathcal{F}$ and $G(g)\cap K_{1}B \neq \varnothing$ define
\[R_{k}(B)=B'\cap \{|g-a_{m_{k}}|>t/4\} \subset B(c,s).\]

Note it will be clear from the context which map $g$ is meant when $R_{k}(B)$ appears.

\begin{claim}\label{disjoint}
If $k\geq 1$ is fixed and $B_{1}, B_{2} \in \mathcal{F}_{k}$ are distinct balls with $G(g)\cap K_{1}B_{1}\neq \varnothing$ and $G(g)\cap K_{1}B_{2} \neq \varnothing$ then $R_{k}(B_{1})\cap R_{k}(B_{2})=\varnothing$.
\end{claim}

\begin{proof}[Proof of Claim \ref{disjoint}]
Let $B_{1}=B(x, t_{1})$ and $B_{2}=B(y,t_{2}) \in \mathcal{F}_{k}$. 

If there exists $l\geq 1$ such that both $B(x',t_{1})$ and $B(y',t_{2})$ belong to $\mathcal{G}_{k}^{l}$ then, since $\{(1/\varepsilon_{k}^{3})B: B\in \mathcal{G}_{k}^{l}\}$ is a disjoint family, $B_{1}'\cap B_{2}'=\varnothing$ so
\[R_{k}(B_{1})\cap R_{k}(B_{2})=\varnothing.\]

Suppose not and $t_{1}\geq t_{2}$. Then, since radii of balls in different families $\mathcal{G}_{k}^{l}$ differ by factor at least $\varepsilon_{k}^3$, necessarily $\varepsilon_{k}^{3} t_{1}\geq t_{2}$. Note, from the definition of $\mathcal{F}_{k}$ in terms of $\mathcal{G}_{k}$, that if $B(z,t)\in \mathcal{F}_{k}$ then $|z_{n+1}-a_{m_{k}}(z')|=2t$. Since $G(g)\cap K_{1}B_{2}\neq \varnothing$ there exists $\widetilde{y}\in B(y',t_{2})$ such that
\[|g(\widetilde{y})-a_{m_{k}}(\widetilde{y})| \leq 4t_{2}.\]
Hence
\[|g-a_{m_{k}}| \leq 4t_{2} + 2t_{2}'/16 \leq t_{1}/4\]
in $B_{2}'$ since $|\nabla g|\leq 1/32$ and $|\nabla a_{m_{k}}| \leq 1/32$ implies $|\nabla (g-a_{m_{k}})|\leq 1/16$ while $B_{2}'$ has diameter $2t_{2}'$. Hence $R_{k}(B_{1})\cap B_{2}'=\varnothing$ which again implies
\[R_{k}(B_{1})\cap R_{k}(B_{2})=\varnothing\]
as required.
\end{proof}

For each $k\geq 1$ we define two subfamilies of $\mathcal{F}_{k}$.
\[\mathcal{F}_{k}^{u}=\{B\in \mathcal{F}_{k}: G(g)\cap K_{1}B \neq \varnothing, |B| \leq \varepsilon_{k} \mathcal{L}^{n}(R_{k}(B))\}\]
and
\[\mathcal{F}_{k}^{d}=\{B\in \mathcal{F}_{k}: G(g)\cap K_{1}B \neq \varnothing, |B| > \varepsilon_{k} \mathcal{L}^{n}(R_{k}(B))\}.\]
Notice $\{ B \in \mathcal{F}_{k}: G(g)\cap K_{1} B \neq \varnothing\} =\mathcal{F}_{k}^{u}\cup \mathcal{F}_{k}^{d}$.

Intuitively $\mathcal{F}_{k}^{u}$ consists of holes where the surface stays mostly on the same level around the hole. Since there are relatively few holes on any one level we can show the holes in this family are small.

\begin{claim}\label{ubound}
For each $k\geq 1$,
\[\sum_{B \in \mathcal{F}_{k}^{u}} |B|\leq \varepsilon_{k}.\]
\end{claim}

\begin{proof}[Proof of Claim \ref{ubound}]
By definition $|B|\leq \varepsilon_{k} \mathcal{L}^{n}(R_{k}(B))$ for all $B \in \mathcal{F}_{k}^{u}$. Hence, using Claim \ref{disjoint},
\[\sum_{B\in \mathcal{F}_{k}^{u}} |B| \leq \varepsilon_{k}\sum_{B \in \mathcal{F}_{k}^{u}} \mathcal{L}^{n}(R_{k}(B)) \leq \varepsilon_{k} \mathcal{L}^{n}\left( \bigcup_{B \in \mathcal{F}_{k}^{u}} R_{k}(B) \right) \leq \varepsilon_{k}.\]
\end{proof}

Intuitively $\mathcal{F}_{k}^{d}$ consists of holes where the surface goes to other levels around the hole. Using results from Section 4 we can show this forces an increase in relative area.

\begin{claim}\label{dbound}
For each $k \geq 1$ and $B \in \mathcal{F}_{k}^{d}$,
\[|B| \leq C\int_{R_{k}(B)}|\nabla (g-a_{m_{k}})|^{2}.\]
\end{claim}

\begin{proof}[Proof of Claim \ref{dbound}]
Notice if $B=B(x,t) \in \mathcal{F}_{k}^{d}$ then
\[\frac{\mathcal{L}^{n}(B'\cap \{|g-a_{m_{k}}|\leq t/4\})}{\mathcal{L}^{n}(B')} = \frac{\mathcal{L}^{n}(B'\setminus R_{k}(B))}{\mathcal{L}^{n}(B')} \geq 1/2.\]
Since $G(g)\cap B \neq \varnothing$ there exists $\widetilde{x} \in B'$ such that $|g(\widetilde{x})-a_{m_{k}}(\widetilde{x})|\geq t/2$. Hence, by Proposition \ref{area} with $h=t/2$,
\[|B|\leq C\int_{R_{k}(B)}|\nabla (g-a_{m_{k}})|^{2}.\]
\end{proof}

\begin{lemma}\label{base}
Fix $g\colon B(c,s) \to \mathbb{R}$ of class $C^{1}$ satisfying $|\nabla g| \leq 1/32$. Then
\[ \sum_{\substack{B \in \mathcal{F}_{1}\\ G(g) \cap K_{1}B \neq \varnothing}} |B| \leq C\int_{B(c,s)}|\nabla g|^{2}+\varepsilon_{1}.\]
\end{lemma}

\begin{proof}[Proof of Lemma \ref{base}]
Since $A_{m_{1}}=A_{1}=p$ and so $a_{m_{1}}=0$ Claim \ref{dbound} implies
\[ \sum_{B \in \mathcal{F}_{1}^{d}} |B| \leq C\int_{\cup_{B \in \mathcal{F}_{1}^{d}} R_{1}(B)}|\nabla g|^{2} \leq C\int_{B(c,s)} |\nabla g|^{2}\]
which, together with Claim \ref{ubound}, proves the lemma.
\end{proof}

\begin{lemma}\label{step}
Suppose we have shown
\[ \sum_{\substack{B \in \mathcal{F}_{1}\cup \ldots \cup \mathcal{F}_{k}\\ G(g) \cap K_{k}B \neq \varnothing}} |B| \leq C\int_{B(c,s)}|\nabla g|^{2}+C\sum_{i=1}^{k} \varepsilon_{i}\]
for all $g\colon B(c,s) \to \mathbb{R}$ of class $C^{1}$ satisfying $|\nabla g|\leq 1/32-3\sum_{i=1}^{k}\varepsilon_{i}$.

Then
\[ \sum_{\substack{B \in \mathcal{F}_{1}\cup \ldots \cup \mathcal{F}_{k+1}\\ G(g) \cap K_{k+1}B \neq \varnothing}} |B| \leq C\int_{B(c,s)}|\nabla g|^{2}+C\sum_{i=1}^{k+1} \varepsilon_{i}\]
for all $g\colon B(c,s) \to \mathbb{R}$ of class $C^{1}$ satisfying $|\nabla g|\leq 1/32-3\sum_{i=1}^{k+1}\varepsilon_{i}$.

Further, the constants can be chosen so as to remain bounded as the lemma is repeatedly applied for all $k \in \mathbb{N}$.
\end{lemma}

\begin{proof}[Proof of Lemma \ref{step}]
Fix $g\colon B(c,s) \to \mathbb{R}$ of class $C^{1}$ satisfying $|\nabla g|\leq 1/32-3\sum_{i=1}^{k+1}\varepsilon_{i}$. We show the measure of holes in $\mathcal{F}_{k+1}$ that the graph of $g$ meets is controlled by the area difference between $g$ and a smoothing $\widetilde{g}$ of $g$. Since the smoothing is only done on small scales the measure of holes in $\mathcal{F}_{1}\cup \ldots \cup \mathcal{F}_{k}$ that $g$ meets will be controlled by the area of $\widetilde{g}$.

\begin{claim}\label{smoothed}
There exists $\widetilde{g}$ of class $C^{1}$ with $|g-\widetilde{g}|\leq \varepsilon_{k+1}r_{k}$ and $|\nabla \widetilde{g}|\leq 1/32-3\sum_{i=1}^{k}\varepsilon_{i}$ such that
\[\sum_{\substack{B \in \mathcal{F}_{k+1}\\ G(g) \cap K_{k+1}B  \neq \varnothing}} |B| \leq C\int_{B(c,s)} (|\nabla g|^{2} - |\nabla \widetilde{g}|^{2})+C\varepsilon_{k+1}.\]
\end{claim}

\begin{proof}[Proof of Claim \ref{smoothed}]
Recall from Claim \ref{ubound},
\[\sum_{B \in \mathcal{F}_{k+1}^{u}} |B|\leq \varepsilon_{k+1}.\]

Since balls in $\{ B':B \in \mathcal{F}_{k+1}\}$ are either disjoint or one is contained inside the other we can choose a subfamily $\mathcal{F}_{k+1}^{s}$ of $\mathcal{F}_{k+1}^{d}$ such that balls in $\{B': B \in \mathcal{F}_{k+1}^{s}\}$ are disjoint and for each ball $B_{1} \in \mathcal{F}_{k+1}^{d}$ there exists $B \in \mathcal{F}_{k+1}^{s}$ such that $B_{1}'\subset B'$.

Notice, by definition of $\mathcal{F}_{k+1}^{d}$, for each ball $B \in \mathcal{F}_{k+1}^{s}$
\begin{align*}
\mathcal{L}^{n}(B'\cap \{|g-a_{m_{k+1}}| > \varepsilon_{k+1}^{3}t'/4\}) &= \mathcal{L}^{n}(B' \cap R_{k+1}(B))\\
&\leq |B|/\varepsilon_{k+1}\\
&\leq \omega_{n} (\varepsilon_{k+1}^{2}t')^{n}
\end{align*}
which, since $|\nabla (g-a_{m_{k+1}})|\leq 1/16$, implies $|g-a_{m_{k+1}}|\leq \varepsilon_{k+1}^{2}t'/2$ in $B'$.

Hence by Claim \ref{dbound} and Proposition \ref{flatten},
\begin{align*}
\sum_{ \substack{B_{1} \in \mathcal{F}_{k+1}^{d} \\ B_{1}' \subset B'}} |B_{1}| &\leq C \sum_{ \substack{B_{1} \in \mathcal{F}_{k+1}^{d} \\ B_{1}' \subset B'}} \int_{R_{k+1}(B_{1})} |\nabla(g-a_{m_{k+1}})|^{2}\\
&\leq C \int_{B'} |\nabla(g-a_{m_{k+1}})|^{2}\\
&\leq C \int_{B'} (|\nabla g|^{2} - |\nabla a_{m_{k+1}}|^{2})+C\varepsilon_{k+1}\mathcal{L}^{n}(B').
\end{align*}

Define $\widetilde{g}$ to be equal to $g$ outside
\[\bigcup_{B \in \mathcal{F}_{k+1}^{s}} B',\]
for $B=B(x,t) \in \mathcal{F}_{k+1}^{s}$ let $\widetilde{g}$ be equal to $g^{\varepsilon_{k+1}t'}$ in $B(x',t'-2\varepsilon_{k+1} t')$ and interpolate smoothly as in Proposition \ref{interpolate} between $g$ and $g^{\varepsilon_{k+1}t'}$  in $B(x',t'-\varepsilon_{k+1}t')\setminus B(x',t'-2\varepsilon_{k+1}t')$.

Since $|g-a_{m_{k+1}}|\leq \varepsilon_{k+1}^{2}t'/2$ in $B'$ and smoothing leaves an affine plane unchanged it follows $|g^{\varepsilon_{k+1}t'}-a_{m_{k+1}}|\leq \varepsilon_{k+1}^{2}t'/2$
in $B(x',t'-\varepsilon_{k+1}t')$. Hence $|g-g^{\varepsilon_{k+1}t'}|\leq \varepsilon_{k+1}^{2}t'$ in $B(x',t'-\varepsilon_{k+1}t')$. 

Also $|\nabla g|\leq 1/32-3\sum_{i=1}^{k+1}\varepsilon_{i}$ implies $|\nabla (g^{\varepsilon_{k+1}t'})| \leq 1/32-3\sum_{i=1}^{k+1}\varepsilon_{i}$ in $B(x',t'-\varepsilon_{k+1}t')$. Proposition \ref{interpolate} hence implies $\widetilde{g}$ is of class $C^{1}$ with $|\nabla \widetilde{g}|\leq 1/32-3\sum_{i=1}^{k}\varepsilon_{i}$. Further, $|g-\widetilde{g}|\leq \varepsilon_{k+1}^{2}t' \leq \varepsilon_{k+1}r_{k}$.

By Proposition \ref{flatten} and Proposition \ref{smoothedarea}, since $|\widetilde{g}-a_{m_{k+1}}|\leq \varepsilon_{k+1}t'$ in $B'$ and $|g-a_{m_{k+1}}|\leq \varepsilon_{k+1}^{2}t'$ in $B(x',t'-\varepsilon_{k+1}t')$,
\begin{align*}
\left| \int_{B'} |\nabla \widetilde{g}|^{2}-\int_{B'}|\nabla a_{m_{k+1}}|^{2}\right| &\leq \int_{B'} |\nabla (\widetilde{g}-a_{m_{k+1}})|^{2}+C\varepsilon_{k+1}\mathcal{L}^{n}(B')\\
&\leq \int_{B(x',t'-2\varepsilon_{k+1}t')} \left| \nabla \left((g-a_{m_{k+1}})^{\varepsilon_{k+1}t'}\right)\right|^{2}+2C\varepsilon_{k+1}\mathcal{L}^{n}(B')\\
&\leq C\varepsilon_{k+1}\mathcal{L}^{n}(B').
\end{align*}

Putting these estimates together and using the definition of $\mathcal{F}_{k+1}^{s}$ gives
\begin{align*}
\sum_{\substack{B \in \mathcal{F}_{k+1}\\ K_{k+1}B\cap G(g) \neq \varnothing}} |B| &\leq \sum_{B \in \mathcal{F}_{k+1}^{s}} \sum_{\substack{B_{1} \in \mathcal{F}_{k+1}^{d}\\ B_{1}' \subset B'}}|B|+\sum_{B \in \mathcal{F}_{k+1}^{u}}|B|\\
&\leq C \sum_{B \in \mathcal{F}_{k+1}^{s}} \left( \int_{B'} (|\nabla g|^{2} - |\nabla a_{m_{k+1}}|^{2}) +C\varepsilon_{k+1} \mathcal{L}^{n}(B')\right) +\varepsilon_{k+1}\\
&\leq C \sum_{B \in \mathcal{F}_{k+1}^{s}} \left( \int_{B'} (|\nabla g|^{2} - |\nabla \widetilde{g}|^{2})+2C\varepsilon_{k+1} \mathcal{L}^{n}(B')\right) +\varepsilon_{k+1}\\
&\leq C \int_{B(c,s)} (|\nabla g|^{2}-|\nabla \widetilde{g}|^{2}) + C\varepsilon_{k+1}
\end{align*}
as required.
\end{proof}

Recall balls in $\mathcal{F}_{1} \cup \ldots \cup \mathcal{F}_{k}$ have radius at least $r_{k}$. Since
\[|g-\widetilde{g}|\leq \varepsilon_{k+1}r_{k}\leq (K_{k}-K_{k+1})r_{k}\]
it follows $G(g)\cap K_{k+1}B \neq \varnothing$ implies $G(\widetilde{g}) \cap K_{k}B \neq \varnothing$ for $B \in \mathcal{F}_{1} \cup \ldots \cup \mathcal{F}_{k}$. Applying the induction assumption to $\widetilde{g}$ and Claim \ref{smoothed} gives,
\begin{align*}
\sum_{\substack{B \in \mathcal{F}_{1} \cup \ldots \cup \mathcal{F}_{k+1}\\ G(g) \cap K_{k+1}B \neq \varnothing}}|B| &\leq \sum_{\substack{B \in \mathcal{F}_{1}\cup \ldots \cup \mathcal{F}_{k}\\ G(\widetilde{g}) \cap K_{k}B \neq \varnothing}}|B|+ \sum_{\substack{B \in \mathcal{F}_{k+1}\\ G(g)\cap K_{k+1}B\neq \varnothing}} |B|\\
& \leq C\int_{B(c,s)} |\nabla \widetilde{g}|^{2} + C\sum_{i=1}^{k}\varepsilon_{i} + C\int_{B(c,s)} (|\nabla g|^{2} - |\nabla \widetilde{g}|^{2})+C\varepsilon_{k+1}\\
& \leq C\int_{B(c,s)} |\nabla g|^{2} + C\sum_{i=1}^{k+1}\varepsilon_{i}.
\end{align*}

It is clear from above that the constants remain bounded as Lemma \ref{step} is repeatedly applied. 
\end{proof}

Since $3\sum_{i=1}^{\infty}\varepsilon_{i}\leq 1/64$ and $K_{k}>1$ for all $k\geq 1$ Lemma \ref{base} and Lemma \ref{step} together prove the proposition.
\end{proof}

Proposition \ref{mainprop} finally proves Theorem \ref{dp} in the case $X=\mathbb{R}^{n+1}$. To see this notice the proposition implies that for $0<r\leq 1/64$
\begin{align*}
\mathcal{H}^{n}(f([0,1]^{n})\cap (B(c,s)\times \mathbb{R})\cap H) &\leq \sum_{\substack{B \in \mathcal{F}\\ f([0,1]^{n})\cap B \neq \varnothing}} \mathcal{H}^{n}(f([0,1]^{n}) \cap B)\\
&\leq Cs^{n}r^{2}+C\sum_{i=1}^{\infty} \varepsilon_{i}
\end{align*}
for all $f \in B_{C^{1}}(p,\delta(r))$. Provided $\varepsilon_{i}$ and $r$ are chosen sufficiently small this implies
\[\mathcal{H}^{n}(f([0,1]^{n})\cap (B(c,s)\times \mathbb{R}) \cap H) \leq \omega_{n}s^{n}/8\]
for all $f \in B_{C^{1}}(p,\delta(r))$ which is \eqref{second} with $\alpha = \omega_{n}s^{n}/2$.

\section{Decomposition of Banach spaces}

We now establish the general case of Theorem \ref{dp} and prove Theorem \ref{sdp}. We have shown there exists $\beta>0$, $0<R<1$ and a directionally porous set $P\subset \mathbb{R}^{n+1}$ such that
\[\{ f \in B_{C^{1}}(p,(n+1)R): \mathcal{H}^{n}(f([0,1]^{n}) \cap P)<(n+1)^{n}\beta\}\]
is meager, where $p\colon [0,1]^{n} \to \mathbb{R}^{n+1}$ is the plane $p(x)=(x,0)$. 

Let $X$ be a (possibly infinite dimensional) Banach space satisfying $\dim X > n$. The main idea used to prove Theorem \ref{dp} and Theorem \ref{sdp} will be that a $C^{1}$ surface of dimension $n$ can be locally approximated by $n$ dimensional affine planes. We use linear maps sending $n$ dimensional affine planes to the plane $p$ in $\mathbb{R}^{n+1}$ to pull back our directionally porous set in $\mathbb{R}^{n+1}$ to directionally porous sets in $X$ then rescale. 

Temporarily fix $\varepsilon>0$, $y \in [0,1-\varepsilon]^{n}$ and $w \in X$. Let $v_{1}, \ldots, v_{n+1}$ be linearly independent vectors in $X$ and $L\colon \mathrm{Span}(v_{1},\ldots,v_{n+1}) \to \mathbb{R}^{n+1}$ be the corresponding bijective linear map sending $v_{1}, \ldots, v_{n+1}$ to the standard basis $e_{1}, \ldots, e_{n+1}$ of $\mathbb{R}^{n+1}$. 

Using the Hahn-Banach theorem we can extend $L$ to a linear map $\widetilde{L}$ defined on $X$ with $\| \widetilde{L}\|\leq (n+1) \|L\|$.

\begin{lemma}
Let $\widetilde{L}$ be as defined above. Then the preimage $\widetilde{L}^{-1}(P)$ of a directionally porous set $P\subset \mathbb{R}^{n+1}$ is a directionally porous set in $X$.
\end{lemma}

\begin{proof}
Suppose $P$ is directionally porous with porosity constant $\rho$. Let $x \in \widetilde{L}^{-1}(P)$. Then $\widetilde{L}(x) \in P$ so there exists $v \in \mathbb{R}^{n+1}$ with $\|v\|=1$ such that for any $\delta > 0$ there exists $t \in \mathbb{R}$ with $|t|<\delta$ such that 
\[B(\widetilde{L}(x)+tv,\rho |t|)\cap P = \varnothing.\]
The previous line implies
\[B(x+tL^{-1}(v),\rho |t|/\|\widetilde{L}\|) \cap \widetilde{L}^{-1}(P) = \varnothing.\]
Hence $\widetilde{L}^{-1}(P)$ is directionally porous.
\end{proof}

Consider the map $T\colon \Gamma_{n}(X) \to \Gamma_{n}(\mathbb{R}^{n+1})$ defined by
\[T(f)(x)=\widetilde{L}\left( \frac{f(y+\varepsilon x)-w}{\varepsilon} \right)\]
for $f \in \Gamma_{n}(X)$ and $x \in [0,1]^{n}$.

The map $f \mapsto T(f)+\widetilde{L}(w)/\varepsilon$ is a continuous linear surjection and hence, by the open mapping theorem, is open. It follows that $T$ is continuous and open. Hence if a set $A \subset \Gamma_{n}(\mathbb{R}^{n+1})$ is meager then the preimage $T^{-1}(A) \subset \Gamma_{n}(X)$ is also meager.

The following lemma follows easily from the definition of $T$ and basic facts about Hausdorff measures. 

\begin{lemma}
Let $q \colon [0,1]^{n} \to X$ be the affine plane
\[q(x)=w+(x_{1}-y_{1})v_{1}+\ldots +(x_{n}-y_{n})v_{n}.\]
Define $M(\varepsilon,y,w,v_{1},\ldots,v_{n+1})$ to be the set of surfaces $f \in \Gamma_{n}(X)$ such that
\[\|f(x) - q(x)\|<\varepsilon R/\|L\| \text{ for } x \in y+[0,\varepsilon]^{n},\]
\[\|\partial f / \partial x_{i}(x) - v_{i}\|<R/\|L\| \text{ for } x \in y+[0,\varepsilon]^{n} \text{ and }i=1, \ldots, n,\]
and
\[\mathcal{H}^{n}(f({y+[0,\varepsilon]^{n}}) \cap (w+\varepsilon \widetilde{L}^{-1}(P)))<\beta \varepsilon^{n}/\|L\|^{n}.\]
Then
\[M(\varepsilon,y,w,v_{1},\ldots,v_{n+1}) \subset T^{-1}\{ f \in B_{C^{1}}(p,(n+1)R): \mathcal{H}^{n}(f([0,1]^{n}) \cap P)<(n+1)^{n}\beta \}\]
and hence $M(\varepsilon,y,w,v_{1},\ldots,v_{n+1})$ is meager in $\Gamma_{n}(X)$.
\end{lemma}

It follows immediately from the lemma that for any choice of $v_{1}, \ldots, v_{n+1}$ linearly independent in $X$ the set $S_{n}(\widetilde{L}^{-1}(P),0)$ is not meager in $\Gamma_{n}(X)$. Since $\widetilde{L}^{-1}(P)$ is directionally porous this establishes the general case of Theorem \ref{dp}.

Now we suppose $X$ is separable with countable dense subset $F \subset X$. Let $D$ be the set of all bijective linear maps corresponding to all linearly independent $(n+1)$-tuples $v_{1}, \ldots, v_{n+1} \in F$ in the way defined earlier. Let
\[Q=\bigcup (w+\varepsilon \widetilde{L}^{-1}(P))\]
where the union is taken over rational $\varepsilon > 0$, $w \in F$ and $L \in D$. Each of the sets $w+\varepsilon \widetilde{L}^{-1}(P)$ is directionally porous and hence $Q$ is $\sigma$-directionally porous. We show $X \setminus Q$ is $\Gamma_{n}$-null. 

It is stated in Lemma 5.3.5 \cite{LPT} that if
\[M'=\{f \in \Gamma_{n}(X):\mathrm{rank}\,df < n \text{ on a set of positive measure}\}\]
then $M'$ is meager.

\begin{lemma}\label{meager}
Define
\[M=M' \cup \bigcup M(\varepsilon,y,w,v_{1},\ldots, v_{n+1})\]
where the union is taken over rational $\varepsilon > 0$, $y \in [0,1-\varepsilon]^{n}\cap \mathbb{Q}^{n}$, $w \in F$ and all $(n+1)$-tuples $v_{1}, \ldots, v_{n+1}$ of linearly independent vectors in $F$. 

Suppose the image of a surface $f \in \Gamma_{n}(X)$ meets $X \setminus Q$ in positive measure. Then $f$ belongs to the meager set $M$.
\end{lemma}

\begin{proof}
Since $f$ is Lipschitz it follows $\mathcal{L}^{n}([0,1]^{n}\setminus f^{-1}(Q))>0$. If $f$ does not belong to $M' \subset M$ then there exists $z \in (0,1)^{n}$ such that $\mathrm{rank}\, df(z)=n$ and $z$ is a Lebesgue density point of $[0,1]^{n}\setminus f^{-1} (Q)$.

If $L_{w_{1}, \ldots, w_{n+1}}: \mathrm{Span}(w_{1}, \ldots, w_{n+1}) \to \mathbb{R}^{n+1}$ denotes the linear map sending linearly independent vectors $w_{1}, \ldots, w_{n+1} \in X$ to $e_{1}, \ldots, e_{n+1} \in \mathbb{R}^{n+1}$ it is not hard to show the map defined by
\[(w_{1}, \ldots, w_{n+1}) \mapsto \|L_{w_{1}, \ldots, w_{n+1}}\|\]
is continuous on the open set on which it is defined.

Since $\mathrm{rank}\, df(z)=n$ we can fix $v_{n+1} \in F$ such that
\[\partial f/\partial x_{1} (z), \ldots, \partial f / \partial x_{n} (z), v_{n+1}\]
are linearly independent. Hence there exists $\eta > 0$ such that if
\[\|v_{i}-\partial f / \partial x_{i}(z)\| < \eta\]
for $i=1, \ldots, n$ then $\eta \leq \|L_{v_{1}, \ldots, v_{n+1}} \| \leq 1/\eta$.

Fix rational $\varepsilon > 0$ to be small enough so that 
\[\|f(x)-f(z)- \partial f / \partial x_{1}(z)(x_{1}-z_{1}) - \ldots - \partial f / \partial x_{n}(z)(x_{n}-z_{n})\|<\varepsilon \eta R/4\]
and
\[\|\partial f / \partial x_{i}(x)-\partial f / \partial x_{i}(z)\|<\eta R/2\]
for $x \in z+[0,2\varepsilon]^{n}\subset [0,1]^{n}$ and $i=1, \ldots, n$ and, using the fact that $z$ is a density point of $[0,1]^{n}\setminus f^{-1}(Q)$,
\[ \mathcal{H}^{n}(f(z+[0,2\varepsilon]^{n}) \cap Q) <\beta \eta^{n}\varepsilon^{n}.\]

Now fix $y \in [0,1-\varepsilon]^{n} \cap \mathbb{Q}^{n}$ sufficiently close to $z$ so that
\[\|f(x)-f(y)- \partial f / \partial x_{1}(z)(x_{1}-y_{1}) - \ldots - \partial f / \partial x_{n}(z)(x_{n}-y_{n})\|<\varepsilon \eta R/2\]
for $x \in y+[0,\varepsilon]^{n} \subset z+[0,2\varepsilon]^{n}$.

There exists $w \in F$ and linearly independent $v_{1}, \ldots, v_{n}\in F$ with
\[ \|v_{i}- \partial f / \partial x_{i}(z)\| < \eta\]
for $i=1, 2, \ldots, n$ such that if $q \colon [0,1]^{n} \to X$ is the affine plane
\[q(x)=w+(x_{1}-y_{1})v_{1}+\ldots +(x_{n}-y_{n})v_{n}\]
then
\[\|f(x) - q(x)\|<\varepsilon \eta R\]
and
\[\|\partial f / \partial x_{i}(x) - v_{i}\|<\eta R\]
for $x \in y+[0,\varepsilon]^{n}$ and $i=1, \ldots, n$.

We note $\eta \leq \|L_{v_{1}, \ldots, v_{n+1}} \| \leq 1/\eta$ and observe
\begin{align*}
\mathcal{H}^{n}(f(y+[0,\varepsilon]^{n}) \cap (w+\varepsilon \widetilde{L}_{v_{1}, \ldots, v_{n+1}}^{-1}(P))) &\leq \mathcal{H}^{n}(f(z+[0,2\varepsilon]^{n}) \cap Q)\\
&<\beta \eta^{n} \varepsilon^{n}\\
&\leq \beta \varepsilon^{n}/\|L_{v_{1}, \ldots, v_{n+1}}\|^{n}.
\end{align*}
Hence $f \in M(\varepsilon,y,w,v_{1},\ldots,v_{n+1})\subset M$ as required.
\end{proof}

Lemma \ref{meager} shows $X \setminus Q$ is $\Gamma_{n}$-null. Since $Q$ is $\sigma$-directionally porous this concludes the proof of Theorem \ref{sdp}.

\end{document}